\documentclass[12 pt, english]{amsart}
\usepackage[english]{babel}
\usepackage[utf8]{inputenc}
\usepackage[T1]{fontenc}
\usepackage[english]{babel}
\usepackage{amsmath}
\usepackage{amssymb}
\usepackage{multirow}
\usepackage{mathtools}
\usepackage{amsthm}
\usepackage[new]{old-arrows}
\theoremstyle{plain}
\usepackage{xfrac}
\usepackage{eucal}
\usepackage{comment}
\usepackage{setspace}
\usepackage{amsfonts}
\usepackage{graphicx}
\usepackage[left=2cm,right=2cm,top=2cm,bottom=2cm]{geometry}
\usepackage{hyperref}
\usepackage{pgf}
\usepackage{tikz}
\usepackage{tkz-fct}
\usepackage[all,cmtip]{xy}
\usepackage{tikz-cd}
\usetikzlibrary{intersections,arrows,chains,matrix,positioning,scopes}

\title{A note on complex Lie Algebras isomorphic to their conjugate}
\author{Cyril Demarche}
\date{October 15, 2026}

\address{Sorbonne Université, Université Paris Cité, CNRS, IMJ-PRG, F-75005 Paris, France}

\newcommand{\R}{\mathbb{R}}
\newcommand{\C}{\mathbb{C}}

\DeclareMathOperator{\GL}{GL}
\DeclareMathOperator{\PGL}{PGL}

\DeclareMathOperator{\Gal}{Gal}

\DeclareMathOperator{\Gm}{\mathbf{G}_m}
\DeclareMathOperator{\Spec}{Spec}
\DeclareMathOperator{\Br}{Br}

\theoremstyle{plain}
\newtheorem{theorem}{Theorem}

\newtheorem{proposition}[theorem]{Proposition}
\newtheorem*{proposition*}{Proposition}
\newtheorem{lemma}[theorem]{Lemma}
\newtheorem{corollary}[theorem]{Corollary}
\theoremstyle{definition}

\theoremstyle{remark}
\newtheorem{remark}[theorem]{Remark}

\begin{document}

\maketitle

\begin{abstract}A real Lie algebra defines by extension of scalars a complex Lie algebra that is isomorphic to its Galois conjugate. In this paper, we are interested in the converse property: given a complex Lie algebra that is isomorphic to its conjugate, is it defined over the real numbers? We prove the existence of a $10$-dimensional nilpotent complex Lie algebra for which the answer is negative, disproving a recent conjecture by Der\'e. In addition, we compute the generic obstruction to this descent problem in terms of Brauer groups.
\end{abstract}

\section{Introduction}

In this note, we are basically interested in the following question: given a complex Lie algebra that is isomorphic to its (Galois) conjugate, is it defined over the field of real numbers?

More generally, if $K$ is a field and $K^s$ a separable closure with Galois group $\Gamma := \Gal(K^s / K)$, given a Lie algebra $L$ over $K^s$ such that for all $\gamma \in \Gamma$, the $K^s$-Lie algebras $L^\gamma$ and $L$ are isomorphic, is the Lie algebra $L$ defined over $K$, i.e. does there exist a $K$-Lie algebra $L_0$ such that $L$ is isomorphic to $L_0 \otimes_K K^s$ over $K^s$?

This question is a very particular case of the much more general problem known as the "field of definition versus field of moduli" problem, which can be stated as follows: given a geometric or algebraic object $X$ defined over $K^s$, a necessary condition for $X$ to be defined over $K$ is that $X$ is isomorphic to all of its Galois conjugates under $\Gamma$. When this condition is satisfied, one says that $K$ is a field of moduli of $X$. The main question is now whether $X$ is defined over $K$ or not. In other words, does there exist $X_0$ defined over $K$ such that $X$ is isomorphic over $K^s$ to the base change of $X_0$? If not, can we explain it by natural (cohomological) obstructions?

Many works were completed about this general question, in several particular geometric cases, such as the case of algebraic curves (Shimura, \cite{S72} and D\`ebes-Emsalem, \cite{DE99}), that of Galois covers of algebraic varieties (D\`ebes-Douai, \cite{DD}) and more recently that of higher dimensional algebraic varieties and abelian varieties (Bresciani-Vistoli, \cite{BV24} and Bragg-Lieblich, \cite{murphy24}). The most recent approach to this problem is formulated in the general language of points in algebraic stacks, and natural obstructions for this problem are gerbes over the field of moduli.

In this paper, we focus on the more algebraic setting of Lie algebras over fields. We begin by recalling a few known results and a conjecture in this context.
First, due to the classification of semisimple Lie algebras, it is clear that any semisimple $K^s$-Lie algebra is defined over $K$.

On the other hand, it is classical that there exist Lie algebras over $\C$ with no real form: for instance, consider the $3$-dimensional complex Lie algebra $L$ with basis $(x,y,z)$ and commutators:
\[
[x,y]=0, \, [x,z]=2x, \, [y,z]=iy \, .
\]
The conjugate $\overline{L}$ is the complex Lie algebra with basis $(x',y',z')$ and commutators:
\[
[x',y']=0, \, [x',z']=2x', \, [y',z']=-iy' \, .
\]
A morphism of complex Lie algebras $\varphi : L \to \overline{L}$ has the form $\varphi(x) = ax'+by'$, $\varphi(y)=cx'+dy'$, $\varphi(z) = e x'+fy'+gz'$ (since the derived Lie algebra of $\overline{L}$ is generated by $x'$ and $y'$), for some $a, b, c, d, e, f, g \in \C$. The conditions $[\varphi(x),\varphi(z)]=2\varphi(x)$ and $[\varphi(y),\varphi(z)]=i\varphi(y)$ imply that $a(g-1)=0$, $d(g+1)=0$, $b(2+ig)=0$ and $c(2g-i)=0$, hence $a=b=c=0$ or $b=c=d=0$, so that in particular $\varphi(x)=0$ or $\varphi(y)=0$, hence $\varphi$ is not an isomorphism. Therefore $L$ is not isomorphic to its conjugate $\overline{L}$, hence it has no real form.

In \cite{dere19}, Conjecture 1 states that any complex Lie algebra isomorphic to its conjugate should be defined over $\R$ and Proposition 5.8 of the same paper proves this conjecture for Lie algebras of dimension at most $4$, and mentions that it holds for all complex nilpotent Lie algebras of dimension at most $7$. 

Concerning the particular case of 2-step nilpotent Lie algebras, a consequence of the work \cite{bddg24} is that any complex $2$-step nilpotent Lie algebra of dimension $\leq 8$ is defined over $\R$. All these positive results are based on the classification of low-dimensional Lie algebras.

The main result of this paper is the following Theorem, which disproves the aforementioned conjecture:
\begin{theorem} \label{thm complex lie}
    There exists a $10$-dimensional 2-step nilpotent complex Lie algebra $\mathfrak{g}$ that is isomorphic to its conjugate and not defined over $\R$.
\end{theorem}

A similar example exists in the algebraic group setting, as explained in Corollary \ref{cor}.

Note that we prove the existence of such a Lie algebra, but we did not manage to construct an \emph{explicit} example (defined by generators and relations for instance).

\begin{remark}
A few months after the release of this paper, Borovoi, de Graaf and Guralnick published the preprint \cite{bor-gra-gur} where they provide in Theorem 1.2 an explicit 10-dimensional example. This paper also mentions that dimension 10 is minimal for such an example to exist in the context of $2$-step nilpotent Lie algebras (see Remark 1.3 in \cite{bor-gra-gur}).
\end{remark}

\paragraph{\bf Acknowledgements:} The author warmly thanks Boris Kunyavskii for several helpful discussions about the field of definition problem, and for asking the aforementioned question for Lie algebras.

\section{Main result}

In order to prove Theorem \ref{thm complex lie}, and even a more general version of it, we first study the moduli space of 2-step nilpotent Lie algebras.

Let $K$ be a characteristic zero field, $n \geq 1$ and $1 \leq k \leq {n \choose 2}$. Define $m := {n \choose 2} - k$. Let $V_n$ be a $K$-vector space of dimension $n$ and recall that there is a natural action of the $K$-group $\GL(V_n)$ on the grassmannian variety $\textup{Gr}(k, \Lambda^2(V_n))$. Let $\textup{Lie}_2(n,m)(K)$ be the category of 2-step nilpotent $K$-Lie algebras $\mathfrak{g}$, with derived Lie algebra $\mathfrak{g}'$, such that $\dim_K(\mathfrak{g}') = m$ and $\dim_K(\mathfrak{g}/\mathfrak{g}') = n$.

Define a functor $\phi : \textup{Gr}(k, \Lambda^2 V_n)(K) \to \textup{Lie}_2(n,m)(K)$ as follows: if $F_k \subset \Lambda^2 V_n$ is a $K$-subspace of dimension $k$, define $\phi(F_k)$ to be the $K$-Lie algebra $V_n \oplus \left(\Lambda^2 V_n / F_k \right)$, with the Lie bracket $[(v,\overline{w}), (v', \overline{w}')] := (0, \overline{v \wedge v'})$, where $\overline{w}$ denotes the image of a vector $w \in \Lambda^2 V_n$ in the quotient $\Lambda^2 V_n / F_k$.

\begin{lemma} \label{lem mod space}
The map $\phi$ induces a bijection 
\[\overline{\phi} : \textup{Gr}(k, \Lambda^2 V_n)(K) / \GL(V_n)(K) \xrightarrow{\sim} \textup{Lie}_2(n,m)(K) / \sim \, ,\] 
where the right-hand side denotes the set of isomorphism classes in the category  $\textup{Lie}_2(n,m)(K)$.
\end{lemma}

\begin{proof}
Let $F, F' \in \textup{Gr}(k, \Lambda^2 V_n)(K)$ and let $\varphi : \phi(F) \xrightarrow{\sim} \phi(F')$ be an isomorphism of Lie algebras.
Then $\varphi$ induces a linear isomorphism $\overline{\varphi}$ between the abelianized Lie algebras $\phi(F)^\textup{ab} = V_n$ and $\phi(F')^\textup{ab} = V_n$, i.e. $\overline{\varphi} \in \GL(V_n)(K)$. Since $\varphi$ is an isomorphism of Lie algebras, the morphism $\Lambda^2\overline{\varphi} \in \GL(\Lambda^2 V_n)(K)$ maps $F$ onto $F'$, hence $F$ and $F'$ are in the same $\GL(V_n)(K)$-orbit in $\textup{Gr}(k, \Lambda^2 V_n)(K)$. On the other hand, if two subspaces in $\textup{Gr}(k, \Lambda^2 V_n)(K)$ are in the same $\GL(V_n)(K)$-orbit, then they define isomorphic Lie algebras. Therefore, we get an injective map \[\overline{\phi} : \textup{Gr}(k, \Lambda^2 V_n)(K) / \GL(V_n)(K) \to \textup{Lie}_2(n,m)(K) / \sim \, .\]
Let us prove that this map is surjective. Let $\mathfrak{g} \in  \textup{Lie}_2(n,m)(K)$. By assumption, $\dim_K(\mathfrak{g}/\mathfrak{g}') = n$, so that one can find a linear isomorphism $\mathfrak{g}/\mathfrak{g}' \cong V_n$. Under this identification, the Lie bracket on $\mathfrak{g}$ defines a surjective linear map $\lambda : \Lambda^2 V_n \to \mathfrak{g}'$. Let $F := \ker(\lambda)$. Then $\dim_K(F)=k$, i.e. $F \in \textup{Gr}(k, \Lambda^2 V_n)$. One checks that the choice of a linear splitting of $\mathfrak{g} \to \mathfrak{g}/\mathfrak{g}'$ defines a Lie algebra isomorphism between $\mathfrak{g}$ and $\mathfrak{g}/\mathfrak{g}' \oplus \mathfrak{g}' \cong V_n \oplus (\Lambda^2 V_n)/F = \phi(F)$, hence the map $\overline{\phi}$ is surjective.
\end{proof}

Let $G$ denote the algebraic $K$-group $\GL(V_n)$ and $\overline{G} := \PGL(V_n)$.

\begin{proposition} \label{prop gen free}
    Assume that $3 \leq k \leq \frac{1}{2} {n \choose 2}$ and $(n,k) \neq (4,3), (5,3), (6,3), (5,4), (5,5)$.
    Then the natural action of $\overline{G}$ on $Y := \textup{Gr}(k, \Lambda^2 V_n)$ is generically free.
\end{proposition}

\begin{proof}
By \cite{ela72}, Lemma 1, the generic stabilizer for the action of $\GL(V_n)/\mu_2 \times \GL_k$ on $(\Lambda^2 V_n) \otimes K^k$ is canonically isomorphic to the generic stabilizer for the action of $\GL(V_n)/\mu_2$ on $\textup{Gr}(k,\Lambda^2 V_n)$. We deduce that the action of $\overline{G} := \PGL(V_n)$ on $\textup{Gr}(k,\Lambda^2 V_n)$ is generically free if the generic stabilizer for the action of $\GL(V_n)/\mu_2 \times \GL_k$ on $\mathbb{P}((\Lambda^2 V_n) \otimes K^k)$ is exactly the center $\Gm \times \Gm$.

Then tables 5 and 6 in \cite{ela72}, together with the second paragraph in \cite{pop78} and table 1 in \cite{pop78}, imply that under the assumptions on $(n,k)$, the action of $\PGL(V_n) \times \PGL_k$ on $\mathbb{P}((\Lambda^2 V_n) \otimes K^k)$ is generically free, hence also that of $\overline{G}$ on $Y := \textup{Gr}(k, \Lambda^2 V_n)$.
\end{proof}

Under the assumptions of the proposition, let $Y := \textup{Gr}(k, \Lambda^2 V_n)$. Then there exists a $\overline{G}$-invariant non-empty Zariski open subset $U \subset Y$ such that the action of $\overline{G}$ on $U$ is free. Then the quotient $X := U/\overline{G}$ is a $K$-variety and the natural morphism $\pi : U \to X$ is a $\overline{G}$-torsor.

In particular, we have a partition
\begin{equation}\label{partition}
X(K) = \bigsqcup_{[\sigma] \in H^1(K,\overline{G})} \pi^\sigma (U^\sigma(K)) \, , 
\end{equation}
where $\pi^\sigma : U^\sigma \to X$ denotes the twist of the torsor $\pi$ by the cocycle $\sigma$.

More precisely, if $X(K) \xrightarrow{\partial} H^1(K, \overline{G})$ denotes the map that associates to a $K$-point $x \in X(K)$ the class of the $\overline{G}$-torsor $\pi^{-1}(x)$, then for any $[\sigma] \in H^1(K,\overline{G})$, a point $x \in X(K)$ is in the image of $\pi^\sigma : U^\sigma(K) \to X(K)$ if and only if $\partial(x) = [\sigma]$. And $\pi : U(K) \to X(K)$ is not surjective if (and only if) there exists a non-trivial $[\sigma] \in H^1(K,\overline{G})$ such that $U^\sigma(K) \neq \emptyset$.

\begin{lemma} \label{lem twist}
    For all $[\sigma] \in \textup{Im}(H^1(K, \GL(V_n)/\mu_2) \to H^1(K, \overline{G}))$, the set $U^\sigma(K)$ is non-empty.
\end{lemma}

\begin{proof}
    By \cite{W82}, Theorem 4, the kernel of the morphism $\GL(V_n) \to \GL(\Lambda^2 V_n)$, defined by $g \mapsto \Lambda^2(g)$, is exactly the subgroup $\mu_2 \subset \GL(V_n)$.
    Therefore, there is a natural commutative diagram of algebraic groups :
    \[
    \xymatrix{
    \GL(V_n) \ar@{->>}[r] \ar[d] & \GL(V_n)/\mu_2 \ar@{->>}[r] \ar@{^{(}->}[d] & \PGL(V_n) \ar@{^{(}->}[d] \\
    \GL(\Lambda^2 V_n) \ar[r]^= & \GL(\Lambda^2 V_n) \ar@{->>}[r] & \PGL(\Lambda^2(V_n)) \, .
    }
    \]
    Let $[\sigma_0] \in H^1(K, \GL(V_n)/\mu_2)$ and denote by $[\sigma]$ its image in $H^1(K, \PGL(V_n))$, $[\sigma_0']$ its image in $H^1(K, \GL(\Lambda^2 V_n))$ and $[\sigma']$ its image in $H^1(K, \PGL(\Lambda^2 V_n))$. Since $\PGL(V_n)$ acts on $U$ and $\PGL(\Lambda^2 V_n)$ acts on $Y$ in a compatible way, one can twist the open immersion $U \to Y$ to get an open immersion $U^{\sigma_0} = U^\sigma \to Y^{\sigma'} = Y^{\sigma_0'}$. By Hilbert 90, the set $H^1(K, \GL(\Lambda^2 V_n))$ is trivial, hence $Y^{\sigma'} = Y^{\sigma_0'} \cong Y$. Therefore $U^\sigma$ is an non-empty open subset of $Y = \textup{Gr}(k, \Lambda^2 V_n)$. But rational points $Y(K)$ are Zariski-dense in the Grassmannian variety $Y$ (since $Y$ is a rational variety), so $U^\sigma(K) \neq \emptyset$.
\end{proof}

Assume now that there exists a central simple $K$-algebra $A$ of period $2$ and index dividing $n$. This is equivalent to saying that $n$ is even and $\textup{Br}(K)$ has non-trivial $2$-torsion, i.e. that $n$ is even and there exists a non-split quaternion algebra over $K$ (see \cite{GSz}, Theorem 1.5.8). Let  $\textup{Br}(K)[2]$ denote the $2$-torsion subgroup of the Brauer group of $K$.

\begin{lemma} \label{lem not surj}
If $n$ is even and $\textup{Br}(K)[2] \neq 0$, then $\pi : U(K) \to X(K)$ is not surjective.
\end{lemma}

\begin{proof}
 The assumption implies that in the following natural commutative diagram (where the horizontal maps are the obvious coboundary maps)
\[
\xymatrix{
H^1(K, \GL(V_n)/\mu_2) \ar@{^{(}->}[r] \ar@{^{(}->}[d] & H^2(K, \mu_2) \ar@{^{(}->}[d] \\
H^1(K, \PGL(V_n)) \ar@{^{(}->}[r] & \textup{Br}(K) = H^2(K, \Gm) \, ,
}
\]
all the maps are injective, and the set $H^1(K, \GL(V_n)/\mu_2)$ is not trivial.
Choose a non-trivial class $[\sigma] \in H^1(K, \GL(V_n)/\mu_2)$. Then the previous Lemma implies that $U^\sigma(K) \neq \emptyset$, which guarantees that $U(K) \to X(K)$ is not surjective (see \eqref{partition}).
\end{proof}

The following result implies Theorem \ref{thm complex lie} about complex and real Lie algebras:

\begin{theorem}\label{thm main}
If $\textup{Br}(K)[2] \neq 0$, there exists a $10$-dimensional 2-step nilpotent Lie algebra $\mathfrak{g}$ over $K^s$ that is isomorphic to all its Galois-conjugates and that is not defined over $K$.
\end{theorem}

\begin{proof}
    Let  $n \geq 4$ be even, $3 \leq k \leq \frac{1}{2} {n \choose 2}$ and $(n,k) \neq (4,3), (6,3)$. By Proposition \ref{prop gen free}, the action of $\overline{G}$ on $Y$ is generically free, and one can construct a $\overline{G}$-torsor $\pi : U \to X$ as above. Lemmas \ref{lem twist} and \ref{lem not surj} imply that the map $\pi : U(K) \to X(K)$ is not surjective. In particular, there exists a point $x \in X(K) \setminus \pi(U(K))$. By construction, there exists $\overline{u} \in U(K^s)$ such that $\pi(\overline{u}) = x$ in $X(K^s)$. By Lemma \ref{lem mod space}, $\overline{u} \in U(K^s) \subset Y(K^s)$ defines a 2-step nilpotent $K^s$-Lie algebra $\mathfrak{g}$ of signature $(n,m)$. Since $\pi(\overline{u}) \in X(K)$, we have that for all $\gamma \in \textup{Gal}(K^s/K)$, $\pi({^\gamma \overline{u}}) = \pi(\overline{u})$ in $X(K^s)$, hence there exists $g_\gamma \in \GL(V_n)(K^s)$ such that ${^\gamma \overline{u}} = g_\gamma \cdot \overline{u}$ in $Y(K^s)$. By Lemma \ref{lem mod space}, it implies that ${^\gamma \mathfrak{g}}$ and $\mathfrak{g}$ are isomorphic as Lie $K^s$-algebras.

    Assume now that $\mathfrak{g}$ is defined over $K$, i.e. that there exists a Lie $K$-algebra $\mathfrak{g}_0$ and a $K^s$-isomorphism $\varphi : \mathfrak{g}_0 \otimes_K K^s \xrightarrow{\sim} \mathfrak{g}$. There exists $y \in Y(K)$ corresponding to the $K$-algebra $\mathfrak{g}_0$. In particular, $y$ is in the $G(K^s)$-orbit of $\overline{u}$ in $Y(K^s)$. Hence $y \in U(K)$ and $\pi(y)=\pi(\overline{u})=x$. Therefore $x \in \pi(U(K))$, which is a contradiction. Hence $\mathfrak{g}$ is not defined over $K$, which concludes the proof.

    Since the dimension of $\mathfrak{g}$ is $n+m$ and $n \geq 4$ is even, one checks easily that the smallest possible dimension for $\mathfrak{g}$ is $10$, corresponding to a Lie algebra of signature $(n,m)=(6,4)$, dual to the case $n=6$ and $k=4$.
\end{proof}

Finally, the classical correspondence between nilpotent Lie algebras and unipotent algebraic groups (see \cite{M17}, Theorem 14.37) gives the following:

\begin{corollary} \label{cor}
Let $K$ be a characteristic zero field such that $\textup{Br}(K)[2] \neq 0$. There exists a $10$-dimensional unipotent algebraic group $G$ over $K^s$ that is isomorphic to all its Galois-conjugates and that is not defined over $K$. This group is a central extension of ${\mathbf{G}_a}^6$ by ${\mathbf{G}_a}^4$.
\end{corollary}

\section{Obstruction for the field of moduli to be a field of definition}

Let us now interpret the previous results in the setting of fields of moduli/definition for those Lie algebras, or equivalently in terms of residual gerbes in a given stack, as in \cite{murphy24}, Appendix A, or in \cite{BV24}, section 3.
By Lemma \ref{lem mod space}, the natural stack of 2-step nilpotent Lie algebras of signature $(n,m)$ is either the quotient stack $\mathcal{X} := [Y/G]$, or the quotient stack $\overline{\mathcal{X}} := [Y/\overline{G}]$, where $Y = \textup{Gr}(k,\Lambda^2 V_n)$, $G = \GL(V_n)$ and $\overline{G} = \PGL(V_n)$. The natural surjective morphism $\pi : G \to \overline{G}$ induces a morphism of stacks $\pi_* : [Y/G] \to [Y/\overline{G}]$ that is a bijection $[Y/G](K^s) \to [Y/\overline{G}](K^s)$ on the sets of isomorphism classes over $K^s$. For a stack $\mathcal{Y}$ over $K$, let $\mathcal{Y}(K)$ denote the set of isomorphism classes of $K$-points in $\mathcal{Y}$. Using Hilbert 90, we have an exact sequence of pointed sets:
\[1 \to Y(K)/G(K) = [Y/G](K) \to [Y/\overline{G}](K) \to H^1(K, \PGL(V_n)) \, .\]

More precisely, the morphism $\pi_* : [Y/G] \to [Y/\overline{G}]$ is a $\Gm$-gerbe (see \cite{stacks}, Tag 06PE and \cite{melo09}, beginning of section 4). For any $K$-point $x \in [Y/\overline{G}](K)$ corresponding to a $\overline{G}$-torsor over $\Spec(K)$ and a $\overline{G}$-equivariant map $\alpha : P \to Y$, the stack fiber $\pi_*^{-1}(x)$ is (canonically isomorphic to) the $\Gm$-gerbe (over $\Spec(K)$) of liftings of $P$ to a $G$-torsor through the map $G \to \overline{G}$ (see \cite{gir}, IV.2.5.8). For any lift $x' \in [Y/G](K^s)$ of $x$, the gerbe $\pi_*^{-1}(x)$ is a substack of $[Y/G]$ containing $x'$, and it is canonically isomorphic to the residual gerbe of $x'$ in $[Y/G]$.

Therefore, if $A$ is a non-split quaternion algebra over $K$, the construction of Theorem \ref{thm main} above provides a point $x \in X(K) = (U/\overline{G})(K)$ that can be seen in $[Y/\overline{G}](K)$ via the natural morphism $U/\overline{G} \to [Y/\overline{G}]$. This point does not lift to $[Y/G](K)$ (since $x$ does not lift to $Y(K)$), and for any lift $x' \in [Y/G](K^s)$, the class in $H^2(K, \Gm)$ of the residual gerbe of $x'$ in $[Y/G]$ is exactly the class of $A$ in $\Br(K)$. Of course, in the stack $[Y/\overline{G}]$, the residual gerbe of $x$ is trivial.

In particular, if $\mathfrak{g}$ is a $K^s$-Lie algebra corresponding to the point $x' \in [Y/G](K^s)$ (i.e. a Lie algebra up to isomorphism, as in the statement of Theorem \ref{thm main}), then the field of moduli of $\mathfrak{g}$ is $K$, while the field of definition of $\mathfrak{g}$ is the quadratic extension $L$ of $K$ that is the minimal splitting field of the algebra $A$. And the non-trivial class of $A$ in $\Br(K)$ is precisely the obstruction for the field of moduli of $\mathfrak{g}$ to be equal to the field of definition of $\mathfrak{g}$. 

Conversely, we have the following result:
\begin{proposition}
Assume that $3 \leq k \leq \frac{1}{2} {n \choose 2}$ and $(n,k) \neq (4,3), (5,3), (6,3), (5,4), (5,5)$. Define $m := {n \choose 2} - k$. 
For a generic $2$-step nilpotent Lie algebra $\mathfrak{g}$ over $K^s$ in $\textup{Lie}_2(n,m)$, if $\mathfrak{g}$ is isomorphic to its Galois conjugates under $\Gamma$, then there is a natural class $A(\mathfrak{g})$ in $\Br(K)$ such that $\mathfrak{g}$ is defined over $K$ if and only if $A(\mathfrak{g})=0$.
\end{proposition}

\begin{proof}
    Consider a non-empty $G$-stable open subvariety $U \subset Y := \textup{Gr}(k, \Lambda^2 V_n)$ given by Proposition \ref{prop gen free}. If $\mathfrak{g}$ corresponds to a point $y \in U(K^s)$ (which is the genericity assumption in the statement), then the image of $y$ by $\pi : U \to X=U/\overline{G}$ is a $K$-point $x$ of $X$. Let $A(\mathfrak{g})$ denote the class in $\Br(K)$ of the $\PGL(V_n)$-torsor over $\Spec(K)$ defined by the fiber $\pi^{-1}(x)$ of $\pi$ at $x$. Then $A(\mathfrak{g})=0$ if and only if $x$ lifts to a point in $U(K)$ (corresponding to a $K$-Lie algebra in $\textup{Lie}_2(n,m)$). By Lemma \ref{lem mod space}, this last condition is equivalent to the fact that $\mathfrak{g}$ is $K^s$-isomorphic to the base change of a $K$-Lie algebra in $\textup{Lie}_2(n,m)$, which concludes the proof.
\end{proof}

\end{document}